\documentclass[10pt]{amsart}

\usepackage{amsfonts,amssymb,amscd,amstext}

\usepackage[charter]{mathdesign}
\usepackage[utf8]{inputenc}

\usepackage{graphicx}
\usepackage[colorlinks=true,linkcolor=blue,citecolor=red]{hyperref}

\renewcommand{\le}{\leqslant}
\renewcommand{\ge}{\geqslant}
\renewcommand{\le}{\leqslant}
\renewcommand{\ge}{\geqslant}
\newcommand{\ptl}{\partial}
\newcommand{\rr}{{\mathbb{R}}}
\newcommand{\la}{\lambda}

\newcommand{\h}{H}
\newcommand{\esf}{\mathbb{S}}

\newcommand{\nn}{\mathbb{N}}

\newcommand{\Sg}{\Sigma}
\newcommand{\sg}{\sigma}
\newcommand{\Om}{\Omega}
\newcommand{\om}{\omega}
\newcommand{\eps}{\varepsilon}

\newcommand{\vol}[1]{|#1|}

\usepackage{color}
\definecolor{grey}{rgb}{.7,.7,.7}

\DeclareMathOperator{\divv}{div}

\newtheorem{theorem}{Theorem}[section]
\newtheorem{proposition}[theorem]{Proposition}
\newtheorem{lemma}[theorem]{Lemma}
\newtheorem{corollary}[theorem]{Corollary}

\theoremstyle{definition}

\newtheorem{remark}[theorem]{Remark}

\theoremstyle{remark}

\numberwithin{equation}{section}

\begin{document}

\title[Large isoperimetric regions in $M\times \rr^k$]{Large isoperimetric regions in the product \\ of a compact manifold with Euclidean space}

\author[M.~Ritor\'e]{Manuel Ritor\'e} \address{Departamento de
Geometr\'{\i}a y Topolog\'{\i}a \\
Universidad de Granada \\ E--18071 Granada \\ Espa\~na}
\email{ritore@ugr.es}
\author[E.~Vernadakis]{Efstratios Vernadakis} \address{Departamento de
Geometr\'{\i}a y Topolog\'{\i}a \\
Universidad de Granada \\ E--18071 Granada \\ Espa\~na}
\email{stratos@ugr.es}

\date{\today}

\thanks{Both authors have been supported by MICINN-FEDER MTM2010-21206-C02-01 and MINECO-FEDER MTM2013-48371-C2-1-P grants, and by Junta de Andaluc\'{\i}a grants FQM-325 and P09-FQM-5088}

\begin{abstract}
Given a compact Riemannian manifold $M$ without boundary, we show that large isoperimetric regions in $M\times\rr^k$ are tubular neighborhoods of $M\times\{x\}$, with $x\in\rr^k$.
\end{abstract}


\maketitle

\thispagestyle{empty}

\bibliographystyle{abbrv}

\section{Introduction}
We consider the \emph{isoperimetric problem} of minimizing perimeter under a given volume constraint in $N=M\times\rr^k$, where $ \rr^k$ is the $k$-dimensional Euclidean space and $M$ is an $m$-dimensional compact Riemannian manifold without boundary. Our main result is the following:

\begin{theorem}
\label{th:main}
Let $M$ be a compact Riemannian manifold. There exists a constant $v_0>0$ such that any isoperimetric region in $M\times\rr^k$ of volume $v\ge v_0$ is a tubular neighborhood of $M\times\{x\}$, with $x\in\rr^k$.
\end{theorem}

This result, in case $k=1$, was first proven by Duzaar and Steffen \cite[Prop.~2.11]{du-st}. As observed by Morgan, an alternative proof for $k=1$ can be given using the monotonicity formula and properties of the isoperimetric profile of $M\times\rr$ (see \cite[Cor.~4.12]{MR3385175} for a proof when $M$ is a convex body). Gonzalo considered the general problem in his Ph.D. Thesis \cite{gonzalo}. In $\mathbb{S}^1\times\rr^k$, the result follows from the classification of isoperimetric regions by Pedrosa and Ritor\'e \cite{MR1757077}. Large isoperimetric regions in asymptotically flat manifolds have been recently characterized by Eichmair and Metzger \cite{MR3127063}. It is worth mentioning that W.-T. Hsiang and W.-Y. Hsiang \cite{MR1010154} completely solved the isoperimetric problem in products of Euclidean and hyperbolic spaces. Morgan \cite{MR2249614}, after Barth\'e \cite{MR1904555}, using results by Ros \cite{MR2167260}, provides a lower bound of the isoperimetric profile of a Riemannian product in terms of concave lower bounds of the isoperimetric profiles of the factors.

In our proof we use symmetrization and show in Corollary~\ref{cor:OmitoTL1} that anisotropic scaling of symmetrized isoperimetric regions of large volume  $L^1$-converge to a tubular neighborhood of $M\times\{0\}$. This convergence is improved in Lemma~\ref{lem:OmitoTHdrf} to Hausdorff convergence of the boundaries using the density estimates on tubes from Lemma~\ref{lem:lerilm42prd}, similar to the ones obtained by Ritor\'e and Vernadakis \cite{MR3335407}. Results of White \cite{MR1305283} and Grosse-Brauckmann \cite{MR1432843} on stable submanifolds then imply that the scaled boundaries are cylinders, see Theorem~\ref{thm:mainstable}. For small dimensions, it is also possible to use a result by Morgan and Ros \cite{MR2652015} to get the same conclusion only using $L^1$-convergence. Once it is shown that the symmetrized set is a tube, it is not difficult to prove that the original isoperimetric region is also a tube.

After the distribution of this manuscript, Gonzalo informed us that he had obtained a proof of Theorem~\ref{th:main} in \cite{1312.6311}. His techniques are different from ours and similar to the ones used in \cite{gonzalo}.

\bigskip

Given a measurable set $E\subset N$, their \emph{perimeter} and \emph{volume} will be denoted by $P(E)$ and $\vol{E}$, respectively. We refer the reader to Maggi's book \cite{MR2976521} for background on finite perimeter sets. The $r$-dimensional Hausdorff measure of a set $E$ will be denoted by $H^r(E)$.

On $M\times\rr^k$ we shall consider

the \emph{anisotropic dilation} of ratio $t>0$  defined by
\begin{equation*}
\varphi_t(p,x)=(p,tx), \qquad (p,x)\in M\times\rr^k.
\end{equation*}
Since the Jacobian of the map $\varphi_t$ is $t^k$, we have
\begin{equation}
\label{eq:volvarphi}
\vol{\varphi_t(E)}=t^k\vol{E},\quad\text{for any measurable set\ } E\subset M\times\rr^k.
\end{equation}
Let $\Sigma\subset M\times\rr^k$ be an $(n-1)$-rectifiable set, where $n=m+k$ is the dimension of $N$. At a regular point $p\in\Sg$, the unit normal $\xi$ can be decomposed as $\xi=av+bw$, with $a^2+b^2=1$, $v$ tangent to $M$ and $w$ tangent to $\rr^k$. Then the Jacobian of $\varphi_t|\Sg$ is  equal to $t^{k-1}(t^2a^2+b^2)^{1/2}$. For $t\ge 1$ we get
\begin{equation}
\label{eq:pervarphi}
t^kH^{n-1}(\Sg)\ge H^{n-1}(\varphi_t(\Sg))\ge t^{k-1} H^{n-1}(\Sg),
\end{equation}
and the reversed inequalities when $t\le 1$. Similar properties hold for the perimeter. Equality holds in the right hand side of \eqref{eq:pervarphi} if and only if $a=0$, or equivalently if and only if $\xi$ is tangent to $\rr^k$.

An open ball in $\rr^k$ of radius $r>0$ and center $x$ will be denoted by $D(x,r)$. If it is centered at the origin, we set $D(r)=D(0,r)$. We shall also denote by $T(x,r)$ the set $M\times D(x,r)$, and by $T(r)$ the set $M\times D(r)$. Observe that $\varphi_t(T(x,r))=T(tx,tr)$ and that $T(x,r)$ is the tubular neighborhood of radius $r>0$ of $M\times\{x\}$.

Given any set $E\subset N$ of finite perimeter, we can replace it by a \emph{normalized} set $\text{sym}\,E$ by requiring $\text{sym}\,E\cap(\{p\}\times\rr^k)=\{p\}\times D(r(p))$, where $H^k(D(r(p))$ is equal to the $H^k$-measure of $E\cap (\{p\}\times \rr^k)$. For such a set we get

\begin{theorem}
\label{thm:sym}
\mbox{}
\begin{enumerate}
\item $|\text{\normalfont sym}\,E|=|E|$,
\item $P(\text{\normalfont sym}\, E)\le P(E)$.
\end{enumerate}
\end{theorem}

The proof of Theorem~\ref{thm:sym} is similar to the one of symmetrization in $\rr^n=\rr^m\times\rr^k$ with respect to one of the factors, see Burago and Zalgaller \cite[\S~9]{MR936419}  (or Maggi \cite{MR2976521} for the case $m=1$). The main ingredients are a corresponding inequality for the Minkowski content and approximation of finite perimeter sets by sets with smooth boundary.

Given $E\subset N$, we denote by $E^*$ its orthogonal projection onto $M$. If $E$ is normalized, and $u:E^*\to\rr^+$ measures the radius of the disk obtained projecting $E\cap(\{p\}\times\rr^k)$ to $\rr^k$, we get, assuming enough regularity on $u$, that
\begin{equation*}
\begin{split}
\vol{E}&=\om_k\int_{E^*} u^k dH^m,
\\
H^{n-1}(\ptl E)&= k\om_k\int_{E^*} u^{k-1}\sqrt{1+|\nabla u|^2} dH^m,
\end{split}
\end{equation*}
where $\omega_k=H^k(D(1))$, and $k\omega_k=H^{k-1}(\esf^{k-1})$. The above formulas imply
\begin{equation*}
\label{eq:patube}
\begin{split}
|T(r)|&=\omega_kr^k H^m(M),
\\
P(T(r))&=k\omega_k r^{k-1} H^m(M),
\end{split}
\end{equation*}
so that 
\begin{equation}
\label{eq:profiletubes}
P(T(r))=k\,\big(\omega_kH^m(M)\big)^{1/k}\,|T(r)|^{(k-1)/k}.
\end{equation}

The \emph{isoperimetric profile} of $M\times\rr^k$ is the function $I:(0,+\infty)\to [0,+\infty)$ defined by
\begin{equation*}
I(v)=\inf\{P(E); |E|=v\}.
\end{equation*}
An \emph{isoperimetric region} is a set $E\subset M\times\rr^k$ satisfying $I(|E|)=P(E)$. Existence of isoperimetric regions in $M\times\rr^k$ is guaranteed by a result of Morgan \cite[p.~129]{MR2455580}, since the quotient of $M\times\rr^k$ by its isometry group is compact. From his arguments, it also follows that isoperimetric regions are bounded in $M\times\rr^k$ (see also \cite{MR2979606}). From \eqref{eq:profiletubes} we get
\begin{equation}
\label{eq:upperprofile}
I(v)\le k\,\big(\omega_kH^m(M)\big)^{1/k}\,v^{(k-1)/k},
\end{equation}
for any $v>0$.
The regularity of isoperimetric regions in Riemannian manifolds is well-known, see Morgan \cite{MR1997594} and Gonzales-Massari-Tamanini \cite{MR684753}. The boundary is regular except for a singular set of vanishing $H^{n-7}$ measure. The following properties of the isoperimetric profile hold

\begin{proposition}
\label{prp:I_{Cyl} is contin}
The isoperimetric profile $I$ of $M\times\rr^k$ is non-decreasing and continuous.
\end{proposition}

\begin{proof}
Let $v_1<v_2$, and $E\subset N$ an isoperimetric region of volume $v_2$. Let $0<t<1$ so that $\vol{\varphi_t(E)}=v_1$. By \eqref{eq:pervarphi} we have
\[
I(v_1)\le P(\varphi_t(E))\le P(E)=I(v_2).
\]
This shows that $I$ is non-decreasing.

Let us prove now the right-continuity of $I$ at $v$. Consider an isoperimetric region $E$ of volume $v$. Take a smooth vector field $Z$ with support in the regular part of the boundary of $E$ such that $\int_E\divv Z\neq 0$. The flow $\{\varphi_t\}_{t\in\rr}$ of $Z$ satisfies $(d/dt)|_{t=0} |\varphi_t(E)|\neq 0$. Using the Inverse Function Theorem we obtain a smooth family $\{E_w\}$, for $w$ near $v$, with $|E_w|=w$ and $E_v=E$. The function $f(w)=P(E_w)$ satisfies $f\ge I$ and $I(v)=f(v)$. This implies that $I$ is right-continuous at $v$ since, for $v_i\downarrow v$, we have
\[
I(v)=f(v)=\lim_{i\to\infty} f(v_i)\ge \lim_{i\to\infty} I(v_i)\ge I(v),
\]
by the monotonicity of $I$.

To prove the left-continuity of $I$ at $v$ we take a sequence of isoperimetric regions $E_i$ with $v_i=|E_i|\uparrow v$ and we consider balls $B_i$ disjoint from $E_i$ so that $|E_i\cup B_i|=|E_i|+|B_i|$. Then $I(v)\le P(E_i\cup B_i)=I(v_i)+P(B_i)\le I(v)+P(B_i)$ by the monotonicity of $I$, and the left-continuity follows by taking limits since $\lim_{i\to\infty}P(B_i)=0$.
\end{proof}

We shall also use the following well-known isoperimetric inequalities in $M$ and $M\times\rr^k$

\begin{lemma}[{\cite{du-st}}]
\label{lem:I_C(v)>cv}
Given $0<v_0<H^m(M)$, there exists a constant $a(v_0)>0$
such that
\begin{equation*}
H^{m-1}(\ptl E)\ge a(v_0)\,H^m(E)
\end{equation*}

for any set $E\subset M$ satisfying $0<H^m(E)<v_0$.
\end{lemma}

\begin{lemma}
\label{lem:Iv0}
Given $v_0>0$, there exists a constant $c(v_0)>0$ so that
\begin{equation}
\label{eq:Iv0}
I(v)\ge c(v_0)\,v^{(n-1)/n}
\end{equation}
for any $v\in (0,v_0)$.
\end{lemma}

Lemma \ref{lem:Iv0} follows from the facts that $I(v)$ is strictly positive for $v>0$ and asymptotic to the Euclidean isoperimetric profile when $v$ approaches $0$.

\section{Large isoperimetric regions in $M\times\rr^k$}

In this Section we shall prove that normalized isoperimetric regions of large volume, when scaled down to have constant volume $v_0$, have their boundaries uniformly close to the boundary of the normalized tube of volume $v_0$.

If $E\subset N$ is any finite perimeter set and $T(E)$ is the tube with the same volume as $E$, we define
\begin{equation*}
E^-=E\cap T(E),\quad E^+ =E\setminus T(E).
\end{equation*}
Let $t>0$, and $\Om = \varphi_t(E)$. Since $\varphi_t(E^+)=\Om^+$, \eqref{eq:volvarphi} implies
\begin{equation}
\label{eq:fracEOm}
\frac{\vol{E^+}}{\vol{E}}=\frac{\vol{\Om^+}}{\vol{\Om}}. 
\end{equation}
A similar equality holds replacing $E^+$ by $E^-$.

\begin{proposition}
\label{prp:P>cV}
Let $\{E_i\}_{i\in\nn}$ be a sequence of normalized sets with volumes $\vol{E_i}\to\infty$. Let $v_0>0$ and $0<t_i<1$ so that $\vol{\varphi_{t_i}(E_i)}=v_0$ for all $i\in\nn$, and let $T$ be the tube of volume $v_0$ around $M_0$.

If $\varphi_{t_i}(E_i)$ does not converge to $T$ in the $L^1$-topology, then there is a constant $c>0$, only depending on $\{E_i\}_{i\in\nn}$, so that, passing to a subsequence, there holds
\begin{equation}
\label{eq:P>cV}
\h^{n-1}(\ptl E_i)\ge c \vol{E_i}.
\end{equation}
\end{proposition}

\begin{proof}
Assume $T=M\times D(r)$, and set $\Om_i=\varphi_{t_i}(E_i)$. As $\vol{\Om_i}=\vol{T}$, we get $2\,\vol{\Om_i^+}=\vol{\Om_i\triangle T}$ and, since $\vol{\Om_i\triangle T}$ does not converge to $0$, the sequence $\vol{\Om_i^+}$ does not converge to $0$ either. Let $c_1>0$ be a constant so that $\limsup_{i\to\infty} (\vol{\Om_i^+}/|\Om_i|)>c_1$. From \eqref{eq:fracEOm} we obtain
\begin{equation}
\label{eq:P>cV1}
\limsup_{i\to\infty}\frac{\vol{E_i^+}}{\vol{E_i}}>c_1.
\end{equation}

Now we claim that
\begin{equation}
\label{eq:P>cVclaim}
\liminf_{i\to\infty}\h^m((\Om_i\cap\ptl T)^*)<\h^m(M).
\end{equation}
To prove \eqref{eq:P>cVclaim} we argue by contradiction. Assume  that $\liminf_{i\to\infty}\h^m((\Om_i\cap\ptl T)^*)=\h^m(M)$. As $\Om_i$ is normalized, we have $(\Om_i\cap\ptl T)^*\subset (\Om_i\cap T)^*$ and so $(T\setminus\Om_i)\subset (M\setminus (\Om_i\cap\ptl T)^*)\times D(r)$. This implies $\limsup_{i\to\infty}\vol{T\setminus\Om_i}=0$. Since $\vol{\Om_i}=\vol{T}$, we get $\lim_{i\to\infty}\vol{\Om_i\triangle T}=2\,\lim_{i\to\infty}\vol{T\setminus\Om_i}=0$, a contradiction that proves the claim.

Hence there exists $w\in (0, H^m(M))$ so that
\begin{equation}
\label{eq:P>cV2}
\liminf_{i\to\infty}\h^m((\Om_i\cap\ptl T)^*)<w.
\end{equation}
Let $T(r_i)$ be the normalized tube with $\vol{T(r_i)}=\vol{E_i}$. As $\Om_i\cap T=\varphi_{t_i}(E_i\cap T(r_i))$, we have $(E_i\cap\ptl T(r_i))^*=(\Om_i\cap\ptl T)^*$; from \eqref{eq:P>cV2} we get $\liminf_{i\to\infty} H^m((E_i\cap \ptl T(r_i))^*)<w$, and we obtain
\begin{equation}
\label{eq:P>cV3}
\liminf_{i\to\infty}\h^m((E_i\cap\ptl T(s))^*)<w ,\qquad \forall s\ge r_i. 
\end{equation}
This last step to go from the particular $r_i$ to every $s\ge r_i$ is easy to check as, for any normalized set $E=\bigcup_{p\in E^*} (\{p\}\times D(r(p)))$, we have $(E\cap\ptl T(s))^*=\{p\in M : r(p)\ge s\}$, therefore $(E\cap\ptl T(s))^*\subset (E\cap\ptl T(r))^*$ whenever $s\ge r$.

The above arguments imply, replacing the original sequence by a subsequence, that
\begin{equation}
\label{eq:P>cV4}
\vol{E_i^+}>c_1\,\vol{E_i},\qquad H^m((E_i\cap\ptl T(s))^*)<w, \qquad i\in\nn,\, s\ge r_i.
\end{equation}

Let $a=a(w)$ be the constant in Lemma \ref{lem:I_C(v)>cv}. For the elements of the subsequence satisfying \eqref{eq:P>cV4} we have
\begin{equation*}
\begin{split}
\h^{n-1} (\ptl E_i)&\ge\h^{n-1}(\ptl E_i\cap (N\setminus T(r_i)))
\\
&\ge \int_{r_i}^{\infty}\h^{n-2} (\ptl{E_i\cap \ptl T(s)})\,ds
\\
&\ge \int_{r_i}^{\infty}\h^{n-2} (\ptl{(E_i\cap \ptl T(s))})\,ds
\\
&=\int_{r_i}^{\infty}\h^{m-1} (\ptl{(E_i\cap\ptl T(s))^*)}\,\h^{k-1} (\ptl D(s))\,ds
\\
&\ge \int_{r_i}^{\infty}a\,\h^{m} ((E_i\cap\ptl T(s))^*)\,\h^{k-1} (\ptl D(s))\,ds
\\
&=a\int_{r_i}^{\infty}\h^{n-1} (E_i\cap \ptl T(s))\,ds=a\,\vol{E_i^+} > a\,c_1\vol{E_i},
\end{split}
\end{equation*}
thus proving the result. In the previous inequalities we have used the coarea formula for the distance function to $M\times\{0\}$; that $\ptl(E_i\cap\ptl T(s))\subset \ptl E_i\cap\ptl T(s)$, where the first $\ptl$ denotes the boundary operator in $\ptl T(s)$; the fact that for an $O(k)$-invariant set $F$ we have $F\cap\ptl T(s)=(F\cap\ptl T(s))^*\times\ptl D(s)$, and so $H^{r+k-1}(F\cap\ptl T(s))=H^r((F\cap\ptl T(s))^*)\, H^{k-1}(\ptl D(s))$; that $(\ptl(E_i\cap\ptl T(s)))^*=\ptl (E_i\cap\ptl T(s))^*$; and the isoperimetric inequality on $M$ given in Lemma~\ref{lem:I_C(v)>cv}.
\end{proof}

\begin{corollary}
\label{cor:OmitoTL1}
Let $\{E_i\}_{i\in\nn}$ be a sequence of normalized isoperimetric sets with volumes $\lim_{i\to\infty}\vol{E_i}=\infty$. Let $v_0>0$ and $0<t_i<1$ such that $\Om_i=\varphi_{t_i}(E_i)$ has volume $v_0$ for all $i\in\nn$. Then $\Om_i\to T$ in the $L^1$-topology, where $T$ is the tube of volume $v_0$.
\end{corollary}

\begin{proof}
Regularity results for isoperimetric regions imply that $P(E_i)=H^{n-1}(\ptl E_i)$, choosing as representative of every isoperimetric set the closure of the set of density one points. If $\Om_i$ does not converge to $T$ in the $L^1$-topology then, using \eqref{eq:P>cV} in Proposition~\ref{prp:P>cV} and \eqref{eq:upperprofile}, we get
\[
c\,\vol{E_i}\le P(E_i)\le k\,\big(\omega_kH^m(M)\big)^{1/k}\,\vol{E_i}^{(k-1)/k}
\]
for a subsequence, thus yielding a contradiction by letting $i\to\infty$ since $\vol{E_i}\to\infty$.
\end{proof}

Using density estimates, we shall show now that the $L^1$ convergence of the scaled isoperimetric regions can be improved to Hausdorff convergence. 

In a similar way to Leonardi and Rigot \cite[p.~18]{le-ri} (see also \cite{MR3335407} and David and Semmes \cite{MR1625982}), given $E\subset N$, we define a function $h:\rr^k\times(0,+\infty)\to\rr^+$ by
\begin{equation*}
h(x,R)=\frac{\min\big\{ \vol{E\cap T(x,R)},\vol{T(x,R)\setminus E} \big\}}{R^n},
\end{equation*}
for $x\in\rr^k$ and $R>0$. We remark that the quantity $h(x,R)$ is not homogeneous in the sense of being invariant by scaling since $h(x,R)\le \tfrac{1}{2} (k\omega_k H^m(M))\,R^{k-n}$, which goes to infinity when $R$ goes to $0$. When the set $E$ should be explicitly mentioned, we shall write
\begin{equation*}
h(E,x,R)=h(x,R).
\end{equation*}
\begin{lemma}
\label{lem:lerilm42prd}
\mbox{}
Let  $E\subset N$ be an isoperimetric region of volume $v>v_0$.  Let $\tau>1$  such that $\Om=\varphi_\tau^{-1}(E)$ has volume $v_0$. Choose $\eps$ so that
\begin{equation}
\label{eq:epsfine}
0<\eps<\min\bigg\{v_0,\bigg(\frac{c(v_0)\,v_0^{1/k}}{2k\omega_kH^m(M)}\bigg)^n,\bigg(\frac{c(v_0)}{8n}\bigg)^n\bigg\},
\end{equation}
where $c(v_0)$ is as in \eqref{eq:Iv0}.

Then, for any $x\in\rr^k$ and $R\le 1$ so that $h(\Om,x,R)\le\eps$, we get
\begin{equation*}
h(\Om,x,R/2)=0.
\end{equation*}
Moreover, in case $h(\Om,x,R)=\vol{\Om\cap T(x,R))}\,R^{-n}$, we get $\vol{\Om\cap T(x,R/2)}=0$ and, in case $h(\Om,x,R)=\vol{T(x,R)\setminus \Om}\,R^{-n}$, we have $\vol{T(x,R/2)\setminus \Om}=0$.
\end{lemma}

\begin{proof}
Using Lemma~\ref{lem:Iv0} we get a positive constant $c(v_0)$ so that \eqref{eq:Iv0} is satisfied (i.e., $I(w)\ge c(v_0)\,w^{(n-1)/n}$, for all $0\le w\le v_0$).

Assume first that
\[
h(x,R)=h(\Om,x,r)=\frac{\vol{\Om\cap T(x,R)}}{R^n}.
\]
Define
\begin{equation*}
m(r)=\vol{\Om\cap T(x,r)},\quad 0<r\le R.
\end{equation*}
The function $m(r)$ is non-decreasing and, for $r\le R\le 1$, we get
\begin{equation}
\label{eq:m(t)<eps}
m(r)\le m(R)\le\vol{\Om\cap T(x,R)}\le\eps\,R^n\le\eps<v_0
\end{equation}
by \eqref{eq:epsfine}. Hence $v_0-m(r)>0$ for $0<r\le R$.

By the coarea formula, when $m'(r)$ exists, we get
\begin{equation*}
m'(r)=\frac{d}{dr}\int_0^r\h^{n-1}(\Om\cap \ptl T(x,s))\,ds=\h^{n-1}(\Om\cap \ptl T(x,r)).
\end{equation*}
Now define
\begin{equation*}
\la(r)=\frac{v_0^{1/{k}}}{(v_0-m(r))^{1/{k}}}=\frac{v^{1/{k}}}{\vol{E\setminus T(\tau x,\tau r)}^{1/{k}}}\ge 1,
\end{equation*}
and
\begin{equation*}
\Om(r)=\varphi_{\la(r)}(\Om\setminus T(x,r)),
\end{equation*}
so that $\vol{\Om(r)}=\vol{\Om}$. Then
\begin{equation*}
E(r)=\varphi_\tau(\Om(r))=\varphi_{\la(r)}(E\setminus T(\tau x,\tau r)),
\end{equation*}
and $\vol{E(r)}=\vol{E}$. Then, using \eqref{eq:pervarphi} for $\la(r)\ge 1$ and standard properties of finite~perimeter sets \cite[Lemmas~12.22 and 15.12]{MR2976521}, we have
\begin{equation}
\label{eq:I(v)1}
\begin{split}
I(v)&\le P(E(r))\le \la(r)^k\,\big(P(E\setminus T(\tau x,\tau r))\big)
\\
&\le\frac{v_0}{v_0-m(r)}\,\big(P(E)-P(E\cap T(\tau x,\tau r))+2H^{n-1}(E\cap\ptl T(\tau x,\tau r))\big).
\end{split}
\end{equation}
Since $\tau\ge 1$ and $E\cap\ptl T(\tau x,\tau r)$ is part of a cylinder, using \eqref{eq:pervarphi} again we get
\begin{align*}
P(E\cap T(\tau x, \tau r)&\ge \tau^{k-1}P(\Om\cap T(x,r))\ge \tau^{k-1}c(v_0)\,m(r)^{(n-1)/n},
\\
H^{n-1}(E\cap\ptl T(\tau x,\tau r))&=\tau^{k-1}H^{n-1}(\Om\cap\ptl T(x,r))=\tau^{k-1} m'(r).
\end{align*}
Replacing these expressions in \eqref{eq:I(v)1}, since $P(E)=I(v)$ and $\tau^kv_0=v$, we have
\begin{equation}
\label{eq:m'(r)1}
\begin{split}
2m'(r)&\ge m(r)^{(n-1)/n}\,\bigg(c(v_0)-\frac{m(r)^{1/n}}{\tau^{k-1}v_0}\,I(v)\bigg)
\\
&\ge m(r)^{(n-1)/n}\,\bigg(c(v_0)-\frac{m(r)^{1/n}}{v_0^{1/k}}\,\frac{I(v)}{v^{(k-1)/k}}\bigg)
\\
&\ge m(r)^{(n-1)/n}\,\bigg(c(v_0)-\frac{\eps^{1/n}}{v_0^{1/k}}\,(k\omega_k H^m(M))\bigg)
\\
&
\ge \frac{c(v_0)}{2}\,m(r)^{(n-1)/n},
\end{split}
\end{equation}
where we have also used $m(r)\le\eps$, \eqref{eq:upperprofile}, and \eqref{eq:epsfine}

If there is $r\in[R/2,R]$ such that $m(r)=0$ then, by the monotonicity of the function $m(r)$, we would conclude $m(R/2)=0$ as well. So we assume $m(r)>0$ in $[R/2,R]$. Then by \eqref{eq:m'(r)1}, we get
\[
\frac{c(v_0)}{4}\le \frac{m'(t)}{m(t)^{(n-1)/n}},\,\qquad \h^1\text{-a.e.}
\]
By \eqref{eq:m(t)<eps} we get $m(R)\le \eps R^n$. Integrating between $R/2$ and $R$,
\[
c(v_0)\,R/8\le n\,(m(R)^{1/{n}}-m(R/2)^{1/{n}})\le n\,m(R)^{1/{n}}\le n\,\eps^{1/n} R.
\]
This is a contradiction, since $\eps<(c(v_0)/8n)^{n}$ by \eqref{eq:epsfine}. So the proof in case $h(x,R)=\vol{\Om\cap T(x,R)}\,R^{-n}$ is completed.

Now we deal with the case $h(x,R)=\vol{T(x,R)\setminus \Om}\,R^{-n}$. Define
\begin{equation*}
m(r)=\vol{T(x,r)\setminus\Om}.
\end{equation*}
Then $m(r)$ is a non-decreasing function and
\begin{equation}
\label{eq:m'(r)2}
m'(r)=H^{n-1}(\Om^c\cap\ptl T(x,r))=\frac{1}{\tau^{k-1}}\,H^{n-1}(E^c\cap\ptl T(\tau x,\tau r)),
\end{equation}
since $E^c\cap\ptl T(\tau x,\tau r)$ is part of a tube. We also have $m(r)\le m(R)\le \eps R^n\le\eps<v_0$ by \eqref{eq:epsfine}. Observe that
\begin{equation}
\label{eq:case21}
P(E\cup T(\tau x,\tau r)\le P(E)-P(T(\tau x,\tau r)\setminus E)+2H^{n-1}(E^c\cap\ptl E(\tau x,\tau r)).
\end{equation}
Since $\varphi_\tau(T(x,r)\setminus \Om)=T(\tau x,\tau r)\setminus E$ and $\tau\ge 1$, we get
\begin{equation}
\label{eq:comp2}
\begin{split}
P(T(\tau x,\tau r)\setminus E)&=P(\varphi_\tau(T(x,r)\setminus\Om))\\
&\ge \tau^{k-1}P(T(x,r)\setminus\Om)\ge \tau^{k-1}\,c(v_0)\,m(r)^{(n-1)/n}.
\end{split}
\end{equation}
Now, using that $I$ is a non-decreasing function we easily obtain $P(E)=I(v)\le I(\vol{E\cup T(\tau x,\tau r)})\le P(E\cup T(\tau x,\tau r))$. We estimate $P(E\cup T(\tau x,\tau r))$ from \eqref{eq:case21}. Using \eqref{eq:comp2} and \eqref{eq:m'(r)2}, we get
\begin{equation}
\begin{split}
I(v)=P(E)\le P(E\cup T(\tau x,\tau r))\le I(v)-\tau^{k-1}c(v_0)\,m(r)^{(k-1)/k}+2\tau^{k-1}m'(r),
\end{split}
\end{equation}

and so
\[
\frac{c(v_0)}{2}\le \frac{m'(r)}{m(r)^{(n-1)/n}},\,\qquad \h^1\text{-a.e.}
\]
By \eqref{eq:m(t)<eps} we get $m(R)\le \eps R^n$. Integrating between $R/2$ and $R$,
\[
c(v_0)\,R/4\le n\,(m(R)^{1/{n}}-m(R/2)^{1/{n}})\le n\,m(R)^{1/{n}}\le n\,\eps^{1/n}R,
\]
and we get a contradiction since by \eqref{eq:epsfine} we have $\eps< (c(v_0)/(8n))^{n}<(c(v_0)/(4n))^n$. This concludes the proof.
\end{proof}
Let $F\subset N$, then we define $F_r=\{x\in N: d(x,F)\le r\}$. We improve now the $L^1$-convergence of normalized isoperimetric regions obtained in Corollary~\ref{cor:OmitoTL1} to Hausdorff convergence of their boundaries
\begin{lemma}
\label{lem:OmitoTHdrf}
Let $\{E_i\}_{i\in\nn}$ be a sequence of isoperimetric sets in $N$ with $\lim_{i\to\infty}\vol{E_i}=\infty$. Let $v_0>0$ and $\{t_i\}_{i\in\nn}$ such that $\lim_{i\to\infty}t_i=0$ and $ \vol{\Om_i}=v_0$ for all $i\in\nn$, where $\Om_i=\varphi_{t_i}(E_i)$. Then for every $r>0$, $\ptl\Om_i\subset (\ptl T)_r$, for large enough $i\in \nn$, where $T$ is the tube of volume $v_0$.
\end{lemma}

\begin{proof}
Since $\vol{\Om_i}=v_0$, using \eqref{eq:epsfine} we can choose a uniform $\eps>0$ so that Lemma~\ref{lem:lerilm42prd} holds with this $\eps$ for all $\Om_i$, $i\in\nn$. This means that, for any $x\in N$ and $0<r\le 1$, whenever $h(\Om_i,x,r)\le\eps$ we get $h(\Om_i,x,r/2)=0$.

As  $\Om_i\to  T$  in $L^1(N)$ by Corollary \ref{cor:OmitoTL1}, we can choose a sequence $r_i\to 0$ so that 
\begin{equation}
\label{eq:EitoE Haus3}
 \vol{\Om_i\,\triangle\, T}<r_i^{n+1}.
\end{equation}
Now fix some $0<r<1$. We reason by contradiction assuming that, for some subsequence, there exist 
\begin{equation}
\label{eq:EisubE Haus1}
 x_i\in \ptl\Om_i\setminus (\ptl T)_{r}.
\end{equation}
We distinguish two cases. 

First case: $x_i\in N\setminus T$, for a subsequence.  Choosing $i$ large enough,  \eqref{eq:EisubE Haus1} implies $T(x_i,r_i)\cap T=\emptyset$ and \eqref{eq:EitoE Haus3} yields

\begin{equation*}
 \vol{\Om_i\cap T(x_i,r_i)}\le  \vol{\Om_i\setminus T}\le  \vol{\Om_i\triangle T}< r_i^{n+1}.
\end{equation*}
So, for $i$ large enough, we get
\[
h(\Om_i,x_i,r_i)=\frac{ \vol{\Om_i\cap T(x_i,r_i)}}{r_i^n}<r_i\le \eps.
\]
By Lemma~\ref{lem:lerilm42prd}, we conclude that $ \vol{\Om_i\cap T(x_i,r_i/2)}=0$, a contradiction.

Second case: $x_i\in T$. Choosing $i$ large enough, \eqref{eq:EisubE Haus1} implies $T(x_i,r_i)\subset T$ and so
\begin{equation*}
  \vol{T(x_i,r_i)\setminus \Om_i }\le  \vol{T\setminus \Om_i },\quad\text{for every}\,\, r_i<r.
\end{equation*}
Then, by \eqref{eq:EitoE Haus3},
we get
\begin{equation*}
 \vol{T(x_i,r_i)\setminus \Om_i }\le  \vol{ T\setminus \Om_i}\le  \vol{\Om_i\triangle T}< r_i^{n+1}.
\end{equation*}
So, for $i$ large enough, we get
\[
h(\Om_i,x_i,r_i)=\frac{ \vol{T(x_i,r_i)\setminus \Om_i }}{r_i^n}<r_i\le \eps.
\]
By Lemma~\ref{lem:lerilm42prd}, we conclude that $\vol{T(x_i,r_i/2)\setminus \Om_i }=0$, and we get again contradiction that proves the Lemma.
\end{proof}

\section{Strict $O(k)$-stability of tubes with large radius}

In this Section we consider the orthogonal group $O(k)$ acting on the product $M\times\rr^k$ through the second factor.

Let $\Sg\subset M\times\rr^k$ be a compact hypersurface with constant mean curvature. It is well-known that $\Sg$ is a critical point of the area functional under volume-preserving deformations, and that $\Sg$ is a second order minimum of the area~under volume-preserving variations if and only if
\begin{equation}
\label{eq:2ndvar}
\int_\Sg \big(|\nabla u|^2-q\,u^2\big)\,d\Sg\ge 0,
\end{equation}
for any smooth function $u:\Sg\to\rr$ with mean zero on $\Sg$. In the above formula $\nabla$ is~the gradient on $\Sg$ and $q$ is the function

$\text{Ric}(\xi,\xi)+|\sg|^2$,
where $|\sg|^2$ is the sum of the squared principal curvatures in $\Sg$, $\xi$ is a unit vector field normal to $\Sg$, and $\text{Ric}$ is the Ricci curvature on $N$. 

A hypersurface satisfying \eqref{eq:2ndvar} is usually called \emph{stable} and condition \eqref{eq:2ndvar} is referred to as \emph{stability condition}. In case $\Sg$ is $O(k)$-invariant we can consider an equivariant stability condition: we shall say that $\Sg$ is \emph{strictly $O(k)$-stable} if there exists a positive constant $\la>0$ such that
\begin{equation*}
\int_\Sg \big(|\nabla u|^2-q\,u^2\big)\,d\Sg\ge\la\int_\Sg u^2\,d\Sg
\end{equation*}
for any $O(k)$-invariant function $u:\Sg\to\rr$ with mean zero.

We consider now the tube $T(r)=M\times D(r)$. The boundary of $T(r)$ is the $O(k)$-invariant cylinder $\Sg(r)=M\times \ptl D(r)$, with $(k-1)$ principal curvatures equal to $1/r$. Hence its mean curvature is equal to $(k-1)/r$ and the squared norm of the second fundamental form satisfies $|\sg|^2=(k-1)/r^2$. The inner unit normal to $\Sg(r)$ is the normal to $\ptl D(r)$ in $\rr^k$ (it is tangent to the factor $\rr^k$). This implies $\text{Ric}(\xi,\xi)=0$.

We have the following result

\begin{lemma}
\label{lem:strictlystable}
The cylinder $\Sg(r)$ is strictly $O(k)$-stable if and only if
\begin{equation*}
r^2>\frac{k-1}{\lambda_1(M)},
\end{equation*}
where $\la_1(M)$ is the first positive eigenvalue of the Laplacian in $M$.
\end{lemma}

\begin{proof}
Let $\Sg=\Sg(r)=M\times D(r)$. Observe that an $O(k)$-invariant function with mean zero on $\Sg$ is determined by a function $u:M\to\rr$ with $\int_Mu\,dM=0$. Hence
\[
\begin{split}
\int_\Sg\big(|\nabla u|^2-q\,u^2\big)\,d\Sg&=k\omega_kr^{k-1}\int_M\big(|\nabla_Mu|^2-\frac{k-1}{r^2}\,u^2\big)\,dM
\\
&
\ge k\omega_kr^{k-1}\bigg(\la_1(M)-\frac{k-1}{r^2}\bigg)\,\int_M u^2\,dM
\\
&=\bigg(\la_1(M)-\frac{k-1}{r^2}\bigg)\,\int_\Sg u^2\,d\Sg.
\end{split}
\]
This proves the Lemma.
\end{proof}

Using results by White \cite{MR1305283} and Grosse-Brauckmann \cite{MR1432843} we get

\begin{theorem}
\label{thm:mainstable}
Let $T$ be a normalized tube so that $\Sg=\ptl T$ is a strictly $O(k)$-stable cylinder. Then there exists $r>0$ so that any $O(k)$-invariant finite perimeter set $E$ with $\vol{E}=\vol{T}$ and $\ptl E\subset T_r$ has larger perimeter than $T$ unless $E=T$.
\end{theorem}

\begin{proof}
Since $\Sg$ is strictly $O(k)$-stable, Grosse-Brauckmann \cite[Lemma~5]{MR1432843} implies that, for some $C>0$, $\Sg$ has strictly positive second variation for the functional
\[
F_C=\text{area}+H\,\text{vol}+\frac{C}{2}\,(\text{vol}-\text{vol}(T))^2,
\]
in the sense that the second variation of $F_C$ in the normal direction of a function $u$ satisfies
\[
\delta^2_u F_C=\int_\Sg\big(|\nabla u|^2-q\,u^2\big)\,d\Sg+C\,\bigg(\int_\Sg u\,d\Sg\bigg)^2\ge \la\int_\Sg u^2\,d\Sg,
\]
for any smooth $O(k)$-invariant function $u$ (see the discussion in the proof of Theorem~2 in Morgan and Ros \cite{MR2652015}). In White's proof of Theorem~3 in \cite{MR1305283} it is observed that a sequence of minimizers of $F_C$ in tubular neighborhoods of radius $1/i$ of $\Sg$ are \emph{almost minimizing}, and hence $C^{1,\alpha}$ submanifolds that converge H\"older differentiably to $\Sg$, contradicting the positivity of the second variation of $\Sg$. Theorem~\ref{thm:sym} implies that the symmetrization of these minimizers are again minimizers. Thus we get a family of $O(k)$-minimizers of $F_C$ converging H\"older differentiably to $\Sg$, thus contradicting the strict $O(k)$-stability of $\Sg$.
\end{proof}

\section{Proof of Theorem \ref{th:main}}

First we claim that there exists $v_0>0$ such that, for any isoperimetric region $E$ of volume $|E|\ge v_0$, the set $\text{sym}\,E$ is a tube.

To prove this, consider a sequence of isoperimetric regions $\{E_i\}_{i\in\nn}$ with $\lim_{i\to\infty}|E_i|=\infty$. We know that $\{\text{sym}\,E_i\}_{i\in\nn}$ are also isoperimetric regions. Let $T=M\times D$ be a strictly $O(k)$-stable tube, that exists by Lemma~\ref{lem:strictlystable}. For large $i$, we scale down the sets $\text{sym}\, E_i$ so that $\Om_i=\varphi_{t_i}^{-1}(\text{sym}\,E_i)$ has the same volume as $T$. As $\text{sym}\,E_i$ is isoperimetric and $t_i>1$, we get from \eqref{eq:upperprofile} and \eqref{eq:pervarphi} that $P(\Om_i)\le P(T)$. By Corollary~\ref{cor:OmitoTL1}, the sets $\{\ptl\Om_i\}_{i\in\nn}$ converge to $\ptl T$ in Hausdorff distance. By Theorem~\ref{thm:mainstable}, $\Om_i=T$ for large $i$ and so $\text{sym}\,E_i$ is a tube. This proves the claim. In particular, $H^m(E\cap(\{p\}\times\rr^k))=H^m(D)$ for any $p\in M$.

Hence the isoperimetric profile satisfies $I(v)=C\,v^{(k-1)/k}$ for the constant $C$ in \eqref{eq:profiletubes} and any  $v\ge v_0$. We conclude that
\begin{equation}
\label{eq:Ivgev0}
I(t^kv)=t^{k-1}I(v), \quad\text{whenever } t^k v\ge v_0.
\end{equation}
Let $E$ be an isoperimetric region with volume $|E|>v_0$, and $t<1$ so that $t^k|E|=v_0$. Then
\[
I(t^k|E|)\le P(\varphi_t(E))\le t^{k-1}P(E)=t^{k-1}I(|E|)
\]
by the inequality corresponding to \eqref{eq:pervarphi} when $t\le 1$. By \eqref{eq:Ivgev0}, equality holds and the unit normal $\xi$ to $\text{reg}(\ptl E)$, the regular part of $\ptl E$, is tangent to the $\rr^k$ factor. This implies that the $m$-Jacobian of the restriction $f$ of the projection $\pi_1:M\times\rr^k\to M$ to the regular part of $\ptl E$ is equal to $1$. By Federer's coarea formula for rectifiable sets \cite[3.2.22]{MR0257325} we get
\[
H^{n-1}(\ptl E)=\int_M H^{k-1}(f^{-1}(p))\,dH^m(p).
\]
Assume that $\text{sym}\, E$ is the tube $T(E)=M\times D$. The Euclidean isoperimetric inequality implies $H^{k-1}(f^{-1}(p))\ge H^{k-1}(\{p\}\times \ptl D)$ and so $H^{n-1}(\ptl E)\ge H^{n-1}(\ptl T(E))$, again by the coarea formula. As $P(E)=P(\text{sym}\,E)=P(T(E))$, we get $H^{k-1}(f^{-1}(p))=H^{k-1}(\ptl D)$ for $H^m$-a.e. $p\in M$ and so $\pi_1^{-1}(p)$ is equal to a disc $\{p\}\times D_p$ for $H^m$- a.e. $p\in M$.

The fact that $\xi$ is tangent to $\rr^k$ in $\text{reg}(\ptl E)$ implies that $\text{reg}(\ptl E)$ is locally a cylinder of the form $U\times S$, where $U\subset M$ is an open set and $S\subset\rr^k$ is a smooth hypersurface. Hence the discs $D_p$ are centered at the same point (i.e., $E$ is the translation of a normalized tube, which proves the theorem).

\begin{remark}
The equivariant version of Theorem~2 in Morgan and Ros \cite{MR2652015}, together with Corollary~\ref{cor:OmitoTL1}, can be used to prove Theorem~\ref{th:main} for small dimensions.

\end{remark}

\bibliography{convex}

\end{document}